\documentclass[12]{amsart}
\usepackage{amsfonts}
\newtheorem{thm}{Theorem}

\newtheorem{cor}[thm]{Corollary}
\newtheorem{lem}[thm]{Lemma}

\newtheorem{que}[thm]{Question}
\newtheorem{prop}[thm]{Proposition}
\newtheorem{exa}[thm]{Example}

\theoremstyle{definition}
\newtheorem{defn}[thm]{Definition}
\usepackage{amsmath,amscd}
\newtheorem{rem}[thm]{Remark}

\begin{document}
 \title{Strongly Clean Matrix Rings over Commutative Rings}
 \author[Fan and Yang]{Lingling Fan $^{\dagger}$ and Xiande Yang$^{\ddagger}$}
\maketitle \centerline{\tiny {$\dagger$Department of Mathematics and
Statistics, Memorial University of Newfoundland, St.John's, Canada
A1C 5S7}} \centerline{\tiny lingling@math.mun.ca}
 \centerline{\tiny {$\ddagger$Department of Mathematics, Harbin Institute of Technology, Harbin, P.R.China, 150001}}
\centerline{\tiny{ xiande.yang@gmail.com }}

 \footnote{Corresponding author: Xiande Yang, Department
of Mathematics, Harbin Institute of Technology, Harbin, China,
150001, Tel:+86-451-83108428, Email address: xiande.yang@gmail.com,
May 3, 2008}

 \begin{abstract} A ring $R$ is called strongly clean if every element of $R$ is the sum of a unit and an idempotent that commute.
 By {\rm SRC} factorization, Borooah, Diesl, and Dorsey
 \cite{BDD051} completely determined when ${\mathbb M}_n(R)$ over a commutative local ring $R$ is strongly clean.
 We generalize the notion of {\rm SRC} factorization to commutative rings, prove that commutative $n$-{\rm SRC}
 rings $(n\ge 2)$ are precisely the commutative local rings over which
${\mathbb M}_n(R)$ is strongly clean, and characterize strong
cleanness of matrices over commutative projective-free rings having
{\rm ULP}. The strongly $\pi$-regular property (hence, strongly
clean property) of ${\mathbb M}_n(C(X,{\mathbb C}))$ with $X$ a {\rm
   P}-space relative to ${\mathbb C}$ is also obtained where $C(X,{\mathbb C})$ is the ring of complex valued continuous
   functions.

 \bigskip

 \end{abstract}
 \bigskip
\section{Introduction}
Let $R$  be an associative ring with identity and $U(R)$ denote the
set of units of $R$. An element $a\in R$ is called \emph{strongly
clean} if $a=e+u$ for some $e^2=e$ and $u\in U(R)$ such that $eu=ue$
and the ring $R$ is a \emph{strongly clean ring} if every element of
$R$ is strongly clean \cite{N99}.

Clearly, local rings are strongly clean. An element $a\in R$ is
\emph{strongly $\pi$-regular} if both chains $Ra \supset Ra^2
\supset \cdots$ and $aR \supset a^2R \supset \cdots$ terminate. $R$
is \emph{strongly $\pi$-regular} if every element of $R$ is strongly
$\pi$-regular \cite{Az54}. Strongly $\pi$-regular elements are
strongly clean \cite{N99}. Hence, strongly $\pi$-regular rings are
strongly clean \cite{BM88, N99}. The authors of \cite{A02} and
\cite{Mc03} proved independently that for a topological space $X$,
$C(X)$ is a strongly clean ring iff $X$ is strongly
zero-dimensional. We proved that $C(X,{\mathbb C}) $) is strongly
clean iff $X$ is strongly zero-dimensional \cite{FY071}. So $C(X)$
and $C(X,{\mathbb C}) $) with $X$ strongly zero-dimensional are
strongly clean. In his foundational paper \cite{N99}, Nicholson
asked if the matrix ring over a strongly clean ring is strongly
clean. Wang and Chen \cite{WC04} answered this question negatively.
Then a natural question arose: When is the matrix ring over a
strongly clean ring strongly clean? For local rings,
 Chen, Yang, and Zhou \cite{CYZ052} characterized when the $2\times 2$
 matrix ring ${\mathbb M}_{2}(R)$ over a commutative local ring $R$
 is strongly clean; Li \cite{Li07} characterized when a single $2\times 2$
 matrix over a commutative local ring is strongly clean; Borooah, Diesl, and Dorsey \cite{BDD051}
 characterized when the matrix ring ${\mathbb M}_{n}(R)$ over a commutative local
 ring $R$ is strongly clean; and recently, Yang and Zhou \cite{YZ072}
characterized when the $2\times 2$
 matrix ring ${\mathbb M}_{2}(R)$ over a local ring $R$ is
 strongly clean. For strongly $\pi$-regular rings, Yang and Zhou
 \cite{YZ071} proved that the matrix rings over some strongly $\pi$-regular rings are
 strongly clean. A completely
regular space $X$ is called a {\rm P}-\emph{space (relative to
${\mathbb R}$)} if every prime ideal in $C(X)$ is maximal
\cite[p.63]{GJ76}. In \cite{FY071}, we found that the matrix ring
over $C(X)$ with $X$ a {\rm P}-space relative to ${\mathbb R}$ is
strongly $\pi$-regular (hence, strongly clean).

 In this paper, we continue the study of when a matrix ring is strongly clean. The authors of \cite{BDD051} defined {\rm
SRC} factorization for a commutative local ring. They proved that
for a commutative local ring $R$, ${\mathbb M}_{n}(R)$ is strongly
clean iff $R$ is an $n$-{\rm SRC} ring and they showed that a matrix
ring over a Henselian ring is strongly clean. The theory of {\rm
SRC} factorization is a useful tool for judging strong cleanness of
matrix rings over commutative local rings. However, the theory is
constraint to commutative local rings. In Section $2$, we generalize
this definition to commutative rings (Definition \ref{defn:1.4}),
get a sufficient but not necessary condition for a matrix ring over
a commutative ring to be strongly clean (Theorem \ref{thm:1.6} and
Example \ref{exa:1.8}), and characterize an $n$-{\rm SRC} ring
(Theorem \ref{thm:13}). After reading an earlier version of this
paper(arXiv:0803.2176v1), Alex Diesl and Tom Dorsey improved upon
our results, and we thank them for giving their permission to
include their results here.  Specifically, Propositions 8, 10, and
13 are due to Diesl and Dorsey, as are Remarks 7(2) and 9, and
Example 15 (generalizing our observation for $n=2$).   Also, their
Lemma 3 refines our Corollary 4, and the proof we give of Corollary
4 is due to them. In Section $3$, we study the strong cleanness of
matrices over the class of commutative projective-free rings having
{\rm ULP} (see Definitions \ref{defn:11} and \ref{defn:19}). The
class of commutative projective-free rings having {\rm ULP} includes
commutative local rings, {\rm PID} (principal ideal domains),
polynomial rings with finitely many indeterminates over a {\rm PID}
(Quillen-Suslin Theorem), and etc.. We characterize when a single
matrix over this class of rings is strongly clean (Theorems
\ref{thm:2.7} and \ref{thm:2.10}). These results can help us to find
all strongly clean matrices over $R$ even if ${\mathbb M}_{n}(R)$ is
not strongly clean. In Section $4$, we find new classes of strongly
clean matrix rings--- matrix rings over $C(X,{\mathbb C})$ are
strongly $\pi$-regular (hence, strongly clean) when $X$ is a {\rm
P}-space relative to ${\mathbb C}$ (see the definition in Section
$4$).

Throughout the paper, when  $R[t]$ is a {\rm UFD} (unique
factorization domain), we let $\gcd(h(t), g(t))$ be the {\it
greatest common divisor} of the polynomials $h(t), g(t) \in
 R[t]$. If $R$ is a field, we require $\gcd(h(t), g(t))$ to be the
 monic greatest common divisor of the polynomials $h(t), g(t) \in
 R[t]$. The symbol {\rm Max}$(R)$ denotes the maximal spectrum of a
 commutative ring  $R$, $J(R)$ denotes the Jacobson radical, and ${\mathbb N}$ denotes the set of positive integers.
\section{Strong cleanness of  ${\mathbb M}_{n}(R)$ over a commutative ring $R$}

The authors of \cite{BDD051} defined {\rm SR} factorization and {\rm
SRC} factorization: Let $R$ be a commutative local ring. A
factorization $h(t)=h_{0}(t)h_{1}(t)$ in $R[t]$ of a monic
polynomial $h(t)$ is said to be an ${\rm SR}$ {\it factorization} if
$h_{0}(t)$ and $h_{1}(t)$ are monic and $h_{0}(0)$ and  $h_{1}(1)
\in U(R)$. The ring $R$ is an $n$-${\rm SR}$ {\it ring} if every
monic polynomial of degree $n$ in $R[t]$ has an ${\rm SR}$
factorization. A factorization $h(t)=h_{0}(t)h_{1}(t)$ in $R[t]$ of
a monic polynomial $h(t)$ is said to be an ${\rm SRC}$ {\it
factorization} if it is an {\rm SR} factorization and $\gcd\left(
\overline{h}_{0}(t), \overline{h}_{1}(t)\right)=1$ in the {\rm PID}
$\bar{R}[t]$ $(=\frac{R}{J(R)}[t])$. The ring $R$ is an $n$-${\rm
SRC}$ {\it ring} if every monic polynomial of degree $n$ in $R[t]$
has an ${\rm SRC}$ factorization. $R$ is an {\rm SRC} {\it ring} if
it is an $n$-{\rm SRC} ring for every $n\in {\mathbb N}$. They
proved that the matrix ring ${\mathbb M}_{n}(R)$ is strongly clean
iff $R$ is an $n$-{\rm SRC} ring.

Recall that, for a commutative ring $R$, a pair of polynomials
$(f_{0}(t), f_{1}(t))$ in $R[t]$ is {\it unimodular} if
$f_{0}(t)R[t]+f_{1}(t)R[t]=R[t]$ or equivalently,
 $f_{0}(t)h_{0}(t)+f_{1}(t)h_{1}(t)=1$ with some $h_{0}(t)$ and $ h_{1}(t)$ in
 $R[t]$. For a commutative
 local ring $R$ and monic polynomials $f_{0}(t)$ and $f_{1}(t)$ in $R[t]$, $\gcd\left(\overline{f}_{0}(t), \overline{f}_{1}(t)\right)=1$ iff
$\left(\overline{f}_{0}(t), \overline{f}_{1}(t)\right)$ is
unimodular in $\bar{R}[t]$ iff $\left(f_{0}(t), f_{1}(t)\right)$ is
unimodular in $R[t]$. $0$ and $1$ are the only idempotents of local
rings. So we generalize above definition to commutative rings.
\begin{defn}\label{defn:1.4} Let $R$ be a commutative ring and let $f(t)\in R[t]$ be a
monic polynomial. A factorization $f(t)=f_{0}(t)f_{1}(t)$ in $R[t]$
is called an {\rm SR} {\it factorization} if $f_{i}(t)$ is monic in
$R[t]$ and $f_{i}(e_i)\in U(R)$ with idempotents $e_0\not=e_1\in R$
$(i=0,1)$. The factorization $f(t)=f_{0}(t)f_{1}(t)$ is called an
{\rm SRC} {\it factorization} if, in addition, $(f_{0}(t),
f_{1}(t))$ is unimodular in $R[t]$. The ring $R$ is called an
$n$-{\rm SR} (resp., $n$-{\rm SRC}) {\it ring} if every monic
polynomial of degree $n$ has an {\rm SR} (resp., {\rm SRC})
factorization.
\end{defn}
\begin{thm}\label{prop:1.9}
 Let $R$ be a commutative ring. Then $R$ is strongly clean iff $R$ is a $1$-{\rm SR} ring iff $R$ is a  $1$-{\rm SRC} ring.
\end{thm}
\begin{proof}  Suppose that $R$ is strongly clean. Let $f(t)=t+a\in R[t]$.
Write $-a=e+u$ where $e^2=e\in R$, $u\in U(R)$, and $eu=ue$. So
$f(e)=-u\in U(R)$. Hence, $f(t)=f_{0}(t)f_{1}(t)$ with
$f_{0}(t)=t+a$ and $f_{1}(t)=1$ is an {\rm SR} factorization.
Obviously, this is also an  {\rm SRC} factorization.

 Suppose that $R$ is a  $1$-{\rm SR} ring. Let $a\in
R$. Then $f(t)=t-a$ has an {\rm SR} factorization  in $R[t]$. It
must be that $f(t)=f_{0}(t)$ or $f(t)=f_{1}(t)$. So there exists
$e^2=e\in R$ such that $f(e)=e-a\in U(R)$.  Thus, $a$ is strongly
clean.\end{proof}
\begin{lem}\label{lem:0.5} Let $R$ be a commutative ring and let
$A \in {\mathbb M}_{n}(R)$. Let $f \in R[t]$ be a monic polynomial
for which $f(A)=0$ (e.g. the characteristic polynomial $\chi_{A}$ of
$A$, by the Cayley-Hamilton Theoerem \cite{Mc84}). If $f(e)$ is a
unit for some idempotent $e \in R$, then $A$ is strongly clean.
\end{lem}
\begin{proof} Let $e$ be such an idempotent. We claim that
$A-eI$ is a unit. Using long division, write $f(t)=(t-e)g(t)+f(e)$.
Then, $0=f(A)=(A-eI)g(A)+f(e)I$. Then, $(A-eI)g(A)f(e)^{-1}=I$, and
we conclude (since the two operators involved commute) that $A-eI$
is invertible. Since $eI$ is a central idempotent of ${\mathbb
M}_{n}(R)$, we conclude that $A=(A-eI)+eI$ is strongly clean.
\end{proof}

Note that the first few lines of work are not needed when
$f=\chi_{A}$, since then $f(e)=\det(eI-A)$, which shows immediately
that $eI-A$ is invertible.

\begin{cor}\label{cor:0.6} Let $R$ be a commutative ring and let
$A \in {\mathbb M}_{n}(R)$. If $f=\chi_{A}$ has an $n$-${\rm SRC}$
factorization, then $A$ is strongly clean.
\end{cor}
\begin{proof} By hypothesis, there exist monic polynomials $f_{0}, f_{1} \in
R[t]$ such that $f=f_{0}f_{1}$ and $(f_{0}, f_{1})$ is unimodular,
and idempotents $e_{0}, e_{1}$ for which $f_{0}(e_{0}),
f_{1}(e_{1})$ are units. Find $g_{0}, g_{1}$ such that $f_{0}g_{0}+
f_{1}g_{1}=1$. By \cite [Lemma 11]{BDD051}, $ker(f_{0}(A))\bigoplus
ker(f_{1}(A))=R^{n}$. It is clear that both $ker(f_{0}(A))$ and
$ker(f_{1}(A))$ are $A$-invariant. Now, $A|_{ker(f_{0}(A))}$
satisfies the polynomial $f_{0}$ and $A|_{ker(f_{1}(A))}$ satisfies
the polynomial $f_{1}$. By Lemma \ref{lem:0.5}, $A|_{ker(f_{0}(A))}$
and $A|_{ker(f_{1}(A))}$ are strongly clean. It follows from
\cite{N99} that $A$ is strongly clean. Indeed, let $\varphi \in
End_{R}(R^{n})$ be the projection of $R^{n}$ onto $ker(f_{0}(A))$,
relative to the direct sum $R^{n}=ker(f_{0}(A))\bigoplus
ker(f_{1}(A))$. Then, $A\varphi=\varphi A$ and $\varphi A$ and
$(1-\varphi)A$ are strongly clean in $\varphi{\mathbb
M}_{n}(R)\varphi$ and $(1-\varphi){\mathbb M}_{n}(R)(1-\varphi)$,
respectively.
\end{proof}
\begin{thm}\label{thm:1.6} If $R$ is an $n$-{\rm SRC} ring, then ${\mathbb M}_{n}(R)$ is strongly clean.
\end{thm}
\begin{proof} For any matrix $A\in {\mathbb M}_{n}(R)$, the characteristic
polynomial, $\chi_{A}(t)$, of $A$ has an $n$-{\rm SRC}
factorization. So $A$ is strongly clean by Corollary \ref{cor:0.6}.
That is, ${\mathbb M}_{n}(R)$ is strongly clean.
\end{proof}
\begin{rem}[On Theorem \ref{thm:1.6}] Being an $n$-{\rm SRC} ring is not necessary for the matrix
ring ${\mathbb M}_{n}(R)$ to be strongly clean (see Example
\ref{exa:1.8}).\end{rem}
\begin{rem}[On Definition \ref{defn:1.4}] $1).$ In Corollary \ref{cor:0.6}, there is no restriction that $e_0\not=e_1$,
but in Definition \ref{defn:1.4}, we require $e_0\not=e_1$. Allowing
the idempotents to agree does not really gain anything, since given
an $n$-{\rm SRC} factorization $f=f_0f_1$ with $e_0=e_1$ and
$f(e_0)\in U(R)$, $f=f\cdot 1$ is an $n$-{\rm SRC} factorization
with respect to $e_0$ and any other idempotent.

$2).$ Logically, allowing idempotents other than $0$ and $1$ to
appear in Definition \ref{defn:1.4} is not as much of a
generalization as we might think. But it can simplify computation.
Recall \cite[Proposition 2]{N99}: If $\{ e_{1}, e_{2}, \cdots,
e_{n}\}$ is a set of complete orthogonal central idempotents, then
$R=\bigoplus_{i=1}^{n} e_{i}R= \bigoplus_{i=1}^{n} e_{i}Re_{i}$, and
$R$ is strongly clean iff $e_{i}Re_{i}$ is strongly clean for $i=1,
\cdots, n$. Observe that, for any idempotent $e\in R$ (with $R$
commutative) and $g(t)\in R[t]$, $g(e)=eg(1)+(1-e)g(0)$, and
moreover, that $eg(1)=eg(e)$. In particular, $g(e)$ is a unit in $R$
iff $eg(1)=eg(e)$ is a unit in the corner ring $eR$ and
$(1-e)g(0)=((1-e)g)(0)$ is a unit in the corner ring $(1-e)R$. Thus,
allowing two idempotents $e_1$ and $e_2$ for the polynomials
$f_0(t)$ and $f_1(t)$ in an ${\rm SR}$ factorization $f=f_0f_1$,
look at the associated four term direct sum decomposition
corresponding to $e_0e_1+e_0(1-e_1)+(1-e_0)e_1+(1-e_0)(1-e_1)=1$. We
get a sum of $f$:
$f=f_0f_1=e_0e_1f_0f_1+e_0(1-e_1)f_0f_1+(1-e_0)e_1f_0f_1+(1-e_0)(1-e_1)f_0f_1$.
$e_0e_1f(t)$ and $e_0e_1g(t)$ are units at the identity of
$e_0e_1R$. $(1-e_0)(1-e_1)f_0(t)$ and $(1-e_0)(1-e_1)f_1(t)$ are
units at $0$ of $(1-e_0)(1-e_1)R$. In the other two factors, one of
$f_0$ and $f_1$ (multiplied with corresponding identity of the
corner rings) is a unit at the corresponding identity and the other
is a unit at $0$. So each component of $f_0$ and $f_1$ has an ${\rm
SR}$ factorization corresponding to the trivial idempotents $0$ and
$``1"$ of the corresponding corner rings.

$3).$ We still call the factorization an ${\rm SR}$ (${\rm SRC}$)
factorization as in \cite{BDD051} because Definition \ref{defn:1.4}
is essentially the same as that in \cite{BDD051} when we deal with
the strong cleanness of matrix rings ${\mathbb M}_n(R)$ with $n\ge
4$ (see Proposition \ref{thm:0.1} and Proposition \ref{thm:0.3}
below) although Definition \ref{defn:1.4} is really a generalization
as Example \ref{rem:0.7} shows.
\end{rem}
\begin{prop}\label{thm:0.1} Let $R$ be a $n$-${\rm
SR}$ ring for some $n\geq 4$. Then $R$ is local.
\end{prop}
\begin{proof} Suppose $e\in R$ is a nontrivial idempotent. Thus, $R$
is a nontrivial direct product, say $R=R_{1}\times R_{2}$, of rings.
Consider the monic polynomial $h(t)=(t^{n-1}(t-1),
t^{n-2}(t-1)^{2})\in R[t]$=$R_1[t]$$\times $$R_2[t]$. Suppose that
$h=fg$ is an $n$-${\rm SR}$ factorization. Write $f=(f_{1}, f_{2})$
and $g=(g_{1}, g_{2})$. Clearly, $f_{1}g_{1}$ is an $n$-${\rm SR}$
factorization of $t^{n-1}(t-1)$ in $R_{1}[t]$ and $f_{2}g_{2}$ is an
$n$-${\rm SR}$ factorization of $t^{n-2}(t-1)^{2}$ in $R_{2}[t]$.

Now, more generally, suppose that $fg$ is an $n$-${\rm SR}$
factorization of $t^{k}(t-1)^{n-k}$, over an arbitrary nonzero
commutative ring $R$. The same is then true passing to a quotient
$F=R/\mathfrak{m}$, where $\mathfrak{m}$ is a maximal ideal. But $F$
is a field, so $F[t]$ is a UFD, and it follows that the image of the
monic polynomial $f$ (resp. $g$) must be $t^{i}(t-1)^{j}$ for some
$i$ and $j$. If $ij\not=0, t^{i}(t-1)^{j}$ annihilates every
idempotent, so in order for $fg$ to be an $n$-${\rm SR}$
factorization, $f$ and $g$ must be, in some order, $t^{i}$ and
$(t-1)^{j}$.

Returning to our previous situation, $f_{1}$ must have degree either
$n-1$ or $1$, whereas $f_{2}$ must have degree either $n-2$ or $2$.
Since $1, 2, n-1, n-2$ are all distinct (since $n\geq 4$), we
conclude that $f$ cannot be monic, and hence $h$ has no $n$-${\rm
SR}$ factorization. We conclude that every idempotent of $R$ is
trivial.

It remains to show that $R$ is local. Observe that the property
$n$-${\rm SR}$ passes to quotient rings. In particular, if $R$ has
$n$-${\rm SR}$, where $n\geq 4$, every quotient of $R$ also has no
nontrivial idempotents. Thus, suppose that $R$ has two distinct
maximal ideals $\mathfrak{m}_{1}$ and $\mathfrak{m}_{2}$. By the
Chinese Remainder Theorem, $R/(\mathfrak{m}_{1}\bigcap
\mathfrak{m}_{2})\cong R/\mathfrak{m}_{1}\times R/\mathfrak{m}_{2}$,
which clearly has nontrivial idempotents. We conclude that
$\mathfrak{m}_{1}=\mathfrak{m}_{2}$. It follows that $R$ has a
unique maximal ideal, so $R$ is local, as desired.
\end{proof}
\begin{rem}\label{rem:0.7} The hypothesis that $n\ge 4$ in Proposition \ref{thm:0.1} is, in fact, necessary. One can show that
$\mathbb{C}\times \mathbb{C}$ is an $n$-{\rm SR} ring for $n=2, 3$.
Other examples include
 $R\times R$ where $R$ is quadratically or cubically closed fields or complete local rings with closed fields as
 quotients.\end{rem}
\begin{prop}\label{prop:0.2} Let $n=2$ or $3$ and let $R$ be a $n$-${\rm SR}$ ring. If $R[t]$ has an
irreducible monic polynomial of degree $n$, then $R$ has only the
trivial idempotents.
\end{prop}
\begin{proof} Let $f\in R[t]$ be irreducible, monic, and
degree $n$. Suppose that $e\in R$ is a nontrivial idempotent: regard
$R$ as the direct product of $eR$ and $(1-e)R$. It follows that
either $ef(t)$ or $(1-e)f(t)$ is an irreducible polynomial, since
otherwise, both factorizations must be into monic polynomials of
degree $1$ and $n-1$, respectively, and we can piece these together
to factor $f$ as a product of a monic degree $1$ and degree $n-1$
polynomial. Without loss of generality, suppose $g(t)=ef(t)$ is
irreducible in $eR[t]$. Consider the monic polynomial
$f^{\prime}=(g(t), t^{n-1}(t-1))\in R[t]$. Any $n$-${\rm SR}$
factorization of $f^{\prime}$ must have first coordinate either
degree $0$ or $n$, since $g$ is irreducible. On the other hand, the
second coordinate, as in the proof of Proposition \ref{thm:0.1},
must have degree $1$ or $n-1$, and it follows as in that proof,
since $0, 1, n-1, n$ are all distinct, that $f^{\prime}$ has no
$n$-${\rm SR}$ factorization. We conclude from this contradiction
that $R$ has no nontrivial idempotents.
\end{proof}
\begin{defn}\label{defn:11} \cite[p.17]{C85} A ring $R$
is called {\it projective-free} if every finitely generated
projective $R$-module is free of unique rank.\end{defn} Camillo and
Yu \cite{CY94} proved that $R$ is semiperfect iff $R$ is {\rm
I}-finite and clean (A ring $R$ is called  {\rm I}-{\it finite } if
$R$ does not have an infinite set of non-zero orthogonal
idempotents). For a projective-free ring, we have the following
result.
\begin{prop}\label{prop:12}
 Let $R$ be a projective-free ring. Then the following
are equivalent:
\begin{enumerate}
\item $R$ is a strongly clean ring.
\item $R$ is a clean ring.
\item $R$ is a local ring.
\item $R$ is an exchange ring.
\item $R$ is a semiperfect ring.

\hspace{-0.5 in} If, in addition, $R$ is commutative, then the above
are equivalent to the following:

\item  $R$ is a $1$-{\rm SR} ring.
\item $R$ is a $1$-{\rm SRC} ring.\end{enumerate}
\end{prop}
\begin{proof} $``({\it 3}) \Rightarrow ({\it 1})\Rightarrow ({\it 2})"$.
This is clear.

 $``({\it 2 })\Rightarrow
({\it 4})"$. This is a well-known result in \cite{N77}.

 $``({\it 4}) \Rightarrow
({\it 3})"$. We prove $R$  has only $0$ and $1$ as its idempotents.
Suppose $e^2=e\in R$. Then $R=Re\oplus R(1-e)$. Since $R$ is
projective-free, we get $Re=0$ or $R(1-e)=0$. So $e=0$ or $e=1$. Now
let $r\notin U(R)$. Then because $R$ is an exchange ring, there
exists $e^2=e$ such that $e\in Rr$ and $1-e\in R(1-r)$. That is,
$1\in Rr$ or $1\in R(1-r)$. But $r\notin U(R)$, so $1\in R(1-r)$.
Similarly, $1\in (1-r)R$. So $1-r\in U(R)$. Therefore, $R$ is local.

$``({\it 3}) \Rightarrow ({\it 5})"$. This is clear.

$``({\it 5})\Rightarrow ({\it 2})".$ This is a result of
\cite{CY94}.

 $``({\it 1})\Leftrightarrow ({\it 6})\Leftrightarrow ({\it 7}) ".$
 This is Theorem \ref{prop:1.9}.
\end{proof}
We have not determined whether the rings in Proposition
\ref{prop:0.2} must be local under the hypothesis. However, the {\rm
SRC} hypothesis forces locality for $n\geq 2$, as the next
proposition shows.

\begin{prop}\label{thm:0.3} Let $R$ be a $n$-${\rm
SRC}$ ring for some $n\geq 2$. Then $R$ is local.
\end{prop}
\begin{proof} By Theorem \ref{thm:1.6}, ${\mathbb M}_{n}(R)$ is strongly
clean. Since it is known that strong cleanness passes to corners,
$R$ must therefore be a strongly clean ring. It will therefore
suffice to show that $R$ has no nontrivial idempotents, since a ring
with only trivial idempotents is strongly clean iff it is local by
Proposition \ref{prop:12}. The result now follows from Theorem
\ref{thm:0.1} for $n\geq 4$. However, we give a different,
elementary argument, that works for all $n\geq 2$, rather than
handing only the cases $n=2$ and $n=3$ separately. Suppose that
$e\in R$ is a nontrivial idempotent. Consider the polynomial
$f(t)=t^{n}-et \in R[t]$. Since $R$ is an $n$-${\rm SRC}$ ring,
there is a factorization $f=f_{0}(t)f_{1}(t)$ of $f(t)$ into monic
polynomials such that $(f_{0}(t), f_{1}(t))$ is unimodular and such
that there are idempotents $e_{0}, e_{1}\in R$ such that
$f_{0}(e_{0}), f_{1}(e_{1})$ are units in $R$. We claim that such a
factorization cannot exist.

A trivial factorization cannot occur, since if $g^{2}=g \in R$, then
$f(g)=g(1-e)$ cannot be a unit (since it annihilates $e\not=0$).
Thus, $f_{0}, f_{1}$ are unimodular polynomials, and each has degree
at least $1$. It will therefore suffice to show that $f$ does not
have a nontrivial factorization as a product of a pair of unimodular
monic polynomials. Indeed, let $f=f_{0}f_{1}$ be such a
factorization. Since $e$ is not a unit, $e \in m$ for some maximal
ideal $m$. Since the images of $f_{0}$ and $f_{1}$ are unimodular in
$(R/m)[t]$, we may assume that $R$ is a field and that $f(t)=t^{n}$.
But $R[t]$ is then a UFD, in which case $f_{0}$ and $f_{1}$, must
be, up to units, each power of $t$, but this forces
$f_{0}R[t]+f_{1}R[t]\subseteq tR[t]$, since $f_{0}$ and $f_{1}$ were
monic polynomials with degree at least $1$. This is a contradiction,
and we conclude that the original strongly clean ring $R$ has only
the trivial idempotents, and hence is local.
\end{proof}
Now we immediately get the following result.
\begin{thm}\label{thm:13} Let $R$ be a commutative ring and $(n\geq 2)$. Then the following are equivalent:\begin{enumerate}\item $R$ is an $n$-${\rm
SRC}$ ring . \item $R$ is a local $n$-${\rm SRC}$ ring.\item $R$ is
local and ${\mathbb M}_{n}(R)$ is strongly clean.\end{enumerate}
\end{thm}
\begin{proof}$``(1)\Leftrightarrow (2)".$ By Proposition
\ref{thm:0.3}.

\hspace{0.25 in}$``(2) \Rightarrow (3)".$ By Theorem \ref{thm:1.6}.

\hspace{0.25 in}$``(3)\Rightarrow (2)".$ By \cite[Corollary
15]{BDD051}.
\end{proof}
But for a commutative ring $R$, being an $n$-{\rm SRC} ring is not a
necessary condition for ${\mathbb M}_{n}(R)$ to be strongly clean.

\begin{exa}\label{exa:1.8} Let $R$ be a Boolean ring with more than $2$ elements.
Then $R$ is not an $n$-{\rm SRC} ring for $n\ge2$ because Boolean
rings other than ${\mathbb Z}/2{\mathbb Z}$ can not be {\rm SR}
rings by Proposition \ref{prop:0.2} since $t^{(n-1)}(t-1)+1$ is
always irreducible for $n\ge2$. But ${\mathbb M}_{n}(R)$ is strongly
clean for any positive integer $n$.
\end{exa}

 We define {\rm SRC} factorizations and {\rm SRC}-rings
on commutative rings and it is clear about commutative {\rm
SRC}-rings by Theorem \ref{thm:13} now. In fact, they can be defined
on non-commutative rings. For example,  the $(\ast)$-factorization
and $(\ast \ast)$-factorization  used to characterize strong
cleanness of ${\mathbb M}_2(R)$ over a local ring $R$ (need not be
commutative) is essentially  the {\rm SR} and {\rm SRC}
factorization for non-commutative case \cite{YZ071}. However, for
the non-commutative case, we know very little.
\section{Strong cleanness of matrices over  projective-free rings  having {\rm ULP}}
 Section $2$ shows that the theory of {\rm SRC} factorization can not give us
 new classes of strongly clean matrix rings except the local ones. However, it can help us to find all strongly clean matrices over projective-free rings having {\rm
 ULP} (see Definition \ref{defn:19}) even though the matrix ring is not strongly clean. This is the topic of Section $3$.

 A matrix $A\in {\mathbb M}_{n}(R)$ is called {\it
singular} if $A$ is non-invertible and {\it nonsingular} if $A$ is
invertible. Here, we give a more detailed definition related to
singularity of a matrix.
\begin{defn}\label{defn:2.1} A singular matrix $A\in {\mathbb M}_{n}(R)$  is called
 {\textit{purely singular}} if $I-A$ is singular or
 {\textit{semi-purely singular}} if $I-A$ is nonsingular. A
nonsingular matrix $A\in {\mathbb M}_{n}(R)$ is called
 {\textit{purely nonsingular}} if $I-A$ is nonsingular or
 {\textit{semi-purely nonsingular}} if $I-A$ is singular.
\end{defn}
\noindent Every matrix belongs to exactly one of the above four
types. All types of matrices are strongly clean except purely
singular ones. So we have the following lemma.
\begin{lem}\label{lem:2.2}
The matrix ring ${\mathbb M}_n(R)$ is strongly clean if and only if
its purely singular matrices are strongly clean.
\end{lem}
\begin{lem}\label{lem:2.3}\cite{YZ072}
Let $R$ be a projective-free ring. Then a purely singular matrix
$T\in {\mathbb M}_n(R)$ is strongly clean  iff  $T$ is similar to
$C=\scriptsize \left(
\begin{array}{cc}
T_{0} & 0\\
 0&T_{1}
\end{array}
\right)\normalsize $ where $T_{0}$ is semi-purely nonsingular and
$T_{1}$ is semi-purely singular.
\end{lem}
By this lemma, we get a necessary condition for a matrix to be
strongly clean when $R$ is  commutative projective-free.
\begin{cor}\label{cor:2.4}
 Let $R$ be a commutative projective-free ring. If $T\in {\mathbb M}_n(R)$ is strongly clean,
then $\chi_{T}(t)$ has an $n$-{\rm SR} factorization.
\end{cor}
\begin{proof} If $T$ is  nonsingular, then
$\chi_{T}(t)=\det(tI-T)=f_{0}(t)f_{1}(t)= \chi_{T}(t)\cdot 1$ with
$f_{0}(t) =\chi_{T}(t)$, $f_{1}(t)=1$, $e_{0}=0$, and $e_{1}=1$ is
an $n$-{\rm SR} factorization. If $T$ is semi-purely singular, then
$\chi_{T}(t)=\det(tI-T)=f_{0}(t)f_{1}(t)=1\cdot \chi_{T}(t)$
 with $f_{0}(t)=1$, $f_{1}(t) =\chi_{T}(t)$, $e_{0}=0$, and $e_{1}=1$ is an $n$-{\rm SR} factorization. If $T$ is purely
singular, then, by Lemma \ref{lem:2.3}, $T$ is similar to
$C=\scriptsize \left(
\begin{array}{cc}
T_{0} & 0\\
 0&T_{1}
\end{array}
\right)\normalsize $ where $T_{0}$ is semi-purely nonsingular and
$T_{1}$ is semi-purely singular. So
$\chi_{T}(t)=\chi_{T_{0}}(t)\cdot \chi_{T_{1}}(t)$ with
$f_{0}(t)=\chi_{T_{0}}(t)$, $f_{1}(t)=\chi_{T_{1}}(t)$, $e_{0}=0$,
and $e_{1}=1$ is an $n$-{\rm SR} factorization.
\end{proof}
\begin{defn}\label{defn:19} A commutative ring $R$ is said to have the {\it
unimodular lifting property} ({\rm ULP} for short) if, for any pair
$(f_{0}(t), f_{1}(t))$ of monic polynomials in $R[t]$, the
unimodularity of
$\left(\overline{f}_{0}(t),\overline{f}_{1}(t)\right)$ in
$\frac{R}{\mathfrak{m}}[t]$ for all $\mathfrak{m}\in \text{Max}(R)$
implies the unimodularity of $(f_{0}(t), f_{1}(t))$ in $R[t]$.
\end{defn}
 A ring $R$ is
{\it semilocal} if $R/J(R)$ is semisimple. A commutative ring is
semilocal iff it has finitely many maximal ideals.
\begin{prop}\label{cor:2.3.6} Commutative semilocal rings have  {\rm ULP}. \end{prop} \begin{proof} Let $R$ be a commutative
semilocal ring. Then $R$ has finitely many maximal ideals, say
${\mathfrak{m_{\text{1}}}}, \cdots, {\mathfrak{m_{\text{n}}}}$. Let
$f_{0}(t), f_{1}(t) \in R[t]$ be monic polynomials and
$\left(\overline{f}_{0}(t),\overline{f}_{1}(t)\right)$ be unimodular
in $\frac{R}{{{\mathfrak{m}}_{k}}}[t]$ for $k=1, 2, \cdots, n$.
Since
$\overline{f}_{0}(t)\frac{R}{{{\mathfrak{m}}_{k}}}[t]+\overline{f}_{1}(t)\frac{R}{{{\mathfrak{m}}_{k}}}[t]=\frac{R}{{{\mathfrak{m}}_{k}}}[t]$,
 we get
$f_{0}(t)R[t]+f_{1}(t)R[t]+{\mathfrak{m}}_{k}[t]=R[t]$. Hence,
$f_{0}(t)a_{k}(t)+f_{1}(t)b_{k}(t)+c_{k}(t)=1$ for some $a_{k}(t),
b_{k}(t)\in R[t]$ and $c_{k}(t)\in{\mathfrak{m}}_{k}[t]$. Therefore,
$$1=\Pi_{k=1}^{n}\left(f_{0}(t)a_{k}(t)+f_{1}(t)b_{k}(t)+c_{k}(t)\right)=f_{0}(t)a^{'}(t)+f_{1}(t)b^{'}(t)+c^{'}(t)$$
for some $a^{'}(t), b^{'}(t)\in R[t]$ and $c^{'}(t)\in J(R)[t]$.
Thus,
$R[t]=f_{0}(t)R[t]+f_{1}(t)R[t]+c^{'}(t)R[t]=f_{0}(t)R[t]+f_{1}(t)R[t]+J(R)R[t]$.
Notice that $\frac{R[t]}{f_{0}(t)R[t]+f_{1}(t)R[t]}$ is a finitely
generated $R$-module and
$J(R)\frac{R[t]}{f_{0}(t)R[t]+f_{1}(t)R[t]}=\frac{J(R)R[t]+f_{0}(t)R[t]+f_{1}(t)R[t]}{f_{0}(t)R[t]+f_{1}(t)R[t]}=\frac{R[t]}{f_{0}(t)R[t]+f_{1}(t)R[t]}$.
So, $f_{0}(t)R[t]+f_{1}(t)R[t]=R[t]$ by Nakayama Lemma.
 Therefore, $(f_{0}(t), f_{1}(t))$ is  unimodular in $R[t]$.
\end{proof}
\begin{cor}\label{cor:2.3.7}Commutative local
rings have  {\rm ULP}.\end{cor}

\begin{prop}
\label{prop:2.3.8}Every {\rm UFD} has {\rm ULP}.
\end{prop}
\begin{proof} Let $f_{0}(t), f_{1}(t) \in R[t]$ be monic polynomials and
$\left(\overline{f_{0}}(t), \overline{f_{1}}(t)\right)$ be
unimodular in $\frac{R}{{{\mathfrak{m}}}}[t]$ for every
${\mathfrak{m}}\in \text{Max}(R)$. Then
$\gcd\left(\overline{f_{0}}(t), \overline{f_{1}}(t)\right)$ $=1$ in
$\frac{R}{{{\mathfrak{m}}}}[t]$. We want to prove that
$\gcd(f_{0}(t), f_{1}(t))$ is a unit in $R[t]$. Suppose
$\gcd(f_{0}(t), f_{1}(t))$ is not a unit.

\textbf{Case 1}. $\gcd(f_{0}(t), f_{1}(t))=m \in R$ but $m\notin
U(R)$.

Then there exists ${\mathfrak{m_{0}}} \in \text{Max}(R)$ such that
$m\in {\mathfrak{m_{0}}}$. So $\gcd\left(\overline{f_{0}}(t),
\overline{f_{1}}(t)\right)$$=\overline{m}=0$ in
$\frac{R}{{{\mathfrak{m_{0}}}}}[t]$. This is a contradiction.

\textbf{Case 2}. $\gcd(f_{0}(t), f_{1}(t))=g(t) \in R[t]$ with
$\deg(g(t))\geq 1$ in $R[t]$.

 Then for any
${\mathfrak{m}}\in \text{Max}(R)$, $\gcd\left(\overline{f_{0}}(t),
\overline{f_{1}}(t)\right)\neq 1$ in $\frac{R}{{\mathfrak{m}}}[t]$
because the coefficient of the leading item of $g(t)$ is a unit.

Hence, $(f_{0}(t), f_{1}(t))$ is unimodular in $R[t]$. \end{proof}

Given a monic polynomial
 $f(t)=t^{n}+a_{n-1}t^{n-1}+\cdots+a_{1}t+a_{0}\in R[t]$, the
matrix
 $C_{f}=\scriptsize\left(
\begin{array}{cccccc}
0&0&0&\cdots&0&-a_{0}\\
1&0&0&\cdots&0&-a_{1}\\
0&1&0&\cdots&0&-a_{2}\\
\vdots&\vdots&\vdots&\vdots&\vdots&\vdots\\
0&0&0&\cdots&0&-a_{n-2}\\
0&0&0&\cdots&1&-a_{n-1}\\
\end{array}
\right)\normalsize$ is called the {\it companion matrix} of $f(t).$

\begin{lem}\label{lem:2.6} \cite[Theorem VII.4.3]{Hu74}
 Let $F$ be a field and $f(t)$ be a monic polynomial in $F[t]$. Then
$f(t)$ is the characteristic
 and minimal polynomial of the companion
matrix $C_{f}$.
\end{lem}
\begin{thm}\label{thm:2.7}
 Let $R$ be a commutative ring having {\rm ULP} and
$f(t)=t^{n}+a_{n-1}t^{n-1}+\cdots+a_{1}t+a_{0}\in R[t]$. Then the
companion matrix $C_{f}$ is strongly clean iff
$\chi_{C_{f}}(t)=f(t)$ has an $n$-{\rm SRC} factorization.
\end{thm}
\begin{proof} ``$\Leftarrow$". By Corollary  \ref{cor:0.6}.

``$\Rightarrow$". The argument of Corollary \ref{cor:2.4} shows that
if $T$ is not purely singular, then $\chi_{T}(t)$ has a {\rm trivial
} {\rm SRC} {\rm {factorization}}, that is, one of the factors is
$1$ and the other is $\chi_{T}(t)$ itself. So we can assume $C_{f}$
is purely singular.
 Then by Lemma \ref{lem:2.3}, there exists $P\in {\mathbb
M}_{n}(R)$ such that $P^{-1}C_{f}P=\scriptsize\left(
\begin{array}{cc}
T_{0} & 0\\
 0&T_{1}
\end{array}
\right)\normalsize$ with $T_{0}$ being $k\times k$ semi-purely
nonsingular matrix and $T_{1}$ being $(n-k)\times (n-k)$ semi-purely
singular matrix where $0<k<n$. Then for every maximal ideal
${\mathfrak{m}}$ in $R$, $\overline{C_{f}}=C_{\overline{f}}\in
{\mathbb M}_{n}(\frac{R}{{\mathfrak m}})$ has $\overline{f}(t)\in
\frac{R}{{\mathfrak m}}[t]$ as the characteristic and minimal
polynomial by Lemma \ref{lem:2.6}. So $\overline{f}(t)=
\chi_{\overline{C_{f}}}$$(t)$ $=\chi_{\overline{T}_{0}}(t)\cdot
\chi_{\overline{T}_{1}}(t)=\det(tI_{k}-\overline{T}_{0})$ $\cdot$
$\det(tI_{n-k}-\overline{T}_{1})$. If
$\gcd(\det(tI_{k}-\overline{T}_{0}),
\det(tI_{n-k}-\overline{T}_{0}))= g(t)$ with degree
$\deg(g(t))\geq1$, then the minimal polynomial of $\overline{C_{f}}$
is $\frac{\det(tI_{k}-\overline{T}_{0})
\det(tI_{n-k}-\overline{T}_{1})}{g(t)}$ which has degree less than
$\deg(\chi_{\overline{C_{f}}})=deg(f)$. This is a contradiction. So
 $f_{0}(t)=\det(tI-T_{0})$, $f_{1}(t)=\det(tI-T_{1})$, $e_{i}=i$, and $f_{i}(e_{i})\in U(R)$ $(i=0,1)$ give an $n$-{\rm SRC}
factorization for $\chi_{C_{f}}(t)=f(t)$.
 \end{proof}
\begin{cor}\label{cor:2.8} Let $R$ be a commutative ring having {\rm
ULP} and let $f(t)\in R[t]$ be a monic polynomial of degree
$\deg(f(t))=n$. Then the following are equivalent:
\begin{enumerate}\item For all $A\in {\mathbb M}_{n}(R)$ with
$\chi_{A}(t)=f(t)$, $A$ is strongly clean.
\item The companion matrix $C_{f}$ is strongly clean.
\item $f(t)$ has an $n$-{\rm SRC} factorization.\end{enumerate}
\end{cor}
\begin{proof} ``$({\it 1}) \Rightarrow ({\it 2})$" is clear.
``$({\it 2}) \Rightarrow ({\it 3})$" is Theorem \ref{thm:2.7}.
``$({\it 3}) \Rightarrow ({\it 1})$" is Corollary \ref{cor:0.6}.
\end{proof}
\begin{que}Does every commutative projective-free ring have {\rm
ULP}?
\end{que}
\begin{thm}\label{thm:2.10} Let $R$ be a commutative projective-free ring. Then a
purely singular matrix $A\in {\mathbb M}_{n}(R)$ is strongly clean
iff $\chi_{A}(t)$ has an $n$-{\rm SR} factorization
$\chi_{A}(t)=f_{0}(t)f_{1}(t)$ with $e_{i}=i$ $(i=0,1)$ and $A$ is
similar to $\scriptsize \left(
\begin{array}{cc}
T_{0} & 0\\
 0&T_{1}
\end{array}
\right)\normalsize$ where $\chi_{T_{0}}(t)=f_{0}(t)$ and
$\chi_{T_{1}}(t)=f_{1}(t)$.
\end{thm}
\begin{proof} ``$\Rightarrow$". By Lemma \ref{lem:2.3}, $A$ is similar
to $\scriptsize \left(
\begin{array}{cc}
T_{0} & 0\\
 0&T_{1}
\end{array}
\right)\normalsize$ where $T_{0}$ is semi-purely nonsingular and
$T_{1}$ is semi-purely singular. By Corollary \ref{cor:2.4},
$\chi_{A}(t)$ has an $n$-{\rm SR} factorization
$\chi_{A}(t)=f_{0}(t)f_{1}(t)$ where $\chi_{T_{0}}(t)=f_{0}(t)$,
$\chi_{T_{1}}(t)=f_{1}(t)$, $e_{i}=i$, $f_{i}(e_{i})\in U(R)$
$(i=0,1)$.

``$\Leftarrow$". By Corollary \ref{cor:2.8}, $T_{0}$ and $T_{1}$ are
strongly clean because $\chi_{T_{0}}(t)=f_{0}(t)$ and
$\chi_{T_{1}}(t)=f_{1}(t)$ have trivial {\rm SRC} factorizations. So
$A$ is strongly clean because the strongly clean property is
invariant under similarity. \end{proof}
\section{Strong cleanness  of ${\mathbb M}_n(C(X,{\mathbb C}))$ with $X$ a {\rm P}-space relative to ${\mathbb C}$}
A topological space $X$ is said to be \emph{completely regular} if
whenever $F$ is a closed set and $x$ is a point in its complement,
there exists a continuous function $f: X\rightarrow [0,1]$ such that
$f(x)=1$ and $f[F]=\{0\}$. Let $C(X)$ (resp., $C(X,{\mathbb C})$)
denote the ring of all real (resp., complex) valued continuous
functions from a completely regular Hausdorff space $X$ to the real
number field ${\mathbb R}$ (resp., complex number field ${\mathbb
C}$). For a function $f\in C(X)$ (or $C(X,{\mathbb C}) $), the set
$z(f)$$=\{x\in X: f(x)=0\}$ is called the \emph {zero-set} of $f$.
An open set $U\subseteq X$ is called \emph{functionally open} if the
complement $X\backslash U$ is a zero-set. A topological space $X$ is
called {\it strongly zero-dimensional} if $X$ is a completely
regular Hausdorff space and every finite functionally open cover
$\{U_{i}\}_{i=1}^{k}$ of the space $X$  has a finite open refinement
$\{V_{i}\}_{i=1}^{m}$ such that $V_{i} \cap V_{j}= \emptyset $ for
any $i\neq j$ \cite{E77}. A completely regular space $X$ is called a
{\rm P}-{\it space relative to} ${\mathbb C}$ if every prime ideal
in $C(X, \mathbb{C})$ is maximal. Matrix rings over $C(X)$ with $X$
a {\rm P}-space relative to ${\mathbb R}$ are strongly $\pi$-regular
\cite{FY071}. In this section, we prove the similar results for
$C(X, \mathbb{C})$ with $X$ a Hausdorff {\rm P}-space relative to
${\mathbb C}$. First, we give some notions. For an ideal $I\leq
C(X,{\mathbb C})$, $z[I]=\{z(f): f\in I\}$. An ideal $I\leq
C(X,{\mathbb C})$ is a {\it $z$-ideal} if $z(g)\in z[I]$ implies
$g\in I$. Let $S$ be a ring and $R$ be a subring of $S$ such that
they share the same identity. The ring $S$ is called a \emph{finite
extension} of $R$ if $S$, as an $R$-module, is generated by a finite
set $X$ of generators.
\begin{thm} \label{thm:3.1} Let $X$ be a Hausdorff {\rm P}-space relative to ${\mathbb C}$.
Then $R=C(X, \mathbb{C})$ is strongly regular. Hence, every finite
extension of $R$ is strongly $\pi$-regular. In particular,
$\mathbb{M}_{n}(R)$ is strongly $\pi$-regular.
\end{thm}
\begin{proof} Suppose $X$ is a {\rm P}-space relative to ${\mathbb C}$.  For $p \in X$, set $O_{p}=\{f \in R): z(f)$ is a
neighborhood of $p \}$ and $ M_{p}=\{f\in R): f(p)=0\}$.  Then
$M_{p}$ is a maximal ideal and  $O_{p}$ is a $z$-ideal in $R$ with
$O_{p}\subseteq M_{p}$.

 Let $A_{p}$ be the family of all zero-sets containing a given point $p$. Then $A_{p}$ is
 the unique $z$-ultrafilter converging to $p$ \cite[p.47]{GJ76}. For any ideal $I$ in $R, z[I]$ is a $z$-filter and if
$I$ is a maximal ideal, then $z[I]$ is a $z$-ultrafilter. Thus
$z[O_{p}]\subseteq z[M_{p}]=A_{p}$. So $M_{p}$ is the only maximal
ideal that contains $O_{p}$. Notice that $z(f^{n})=z(f)$ for any $n
\in \mathbb{N}$. If  $I$ is a $z$-ideal and $f^{n} \in I$ then
$z(f)=z(f^{n}) \in z[I] $ implies $f \in I$. So $I$ is a radical
ideal, that is, $I$ is an intersection of prime ideals containing
$I$. Hence,  $O_{p}$ is an intersection of prime ideals. Since
$M_{p}$ is the only maximal ideal that contains $O_{p}$, $O_{p}
\not=M_{p}$ implies $O_{p}$ is contained in a prime ideal that is
not maximal. However, every prime ideal is maximal if $X$ is a ${\rm
P}$-space relative to ${\mathbb C}$. Hence, $O_{p} =M_{p}$.

 Let $p$ be any point in $z(f)$. Then  $f(p)=0$ implies $f \in
 M_{p}=O_{p}$. Hence, $z(f)$ is open, that is, every zero-set is clopen.
Suppose $I$ is an ideal of $R$ and $z(f) \in z[I]$, then $z(f)=z(g)$
for some $g \in I$. Define $h: X \rightarrow \mathbb{C}$ by $h(x)=0$
if $x \in z(f)$ and $h(x)=\frac{f(x)}{g(x)}$ if $x \not \in z(f)$.
Then $h \in R$ and $f=gh$. Thus, $f \in I$, so $I$ is a $z$-ideal.
Hence, every ideal in $R$ is a $z$-ideal. So every ideal is a
radical ideal.

 Since $f$ and $ f^{2}$ belong to the same prime ideals,
$(f)=\bigcap_{f \in {\mathfrak p}:\text{prime}} {\mathfrak p}
=\bigcap_{f^2 \in {\mathfrak p}:\text{prime}} {\mathfrak
p}=(f^{2})$. So $f=f^{2}f_{0}$ for some $f_{0} \in R$.  So $R$ is
strongly regular. Hence, by \cite[Corollary 4]{H90}, every finite
extension of $R$ is strongly $\pi$-regular. In particular,
$\mathbb{M}_{n}(R)$ is strongly $\pi$-regular since
$\mathbb{M}_{n}(R)$ is the finite extension of $R$.
\end{proof}

\begin{cor}Let $X$ be a {\rm P}-space relative to ${\mathbb C}$ and $G$ be a locally finite group and let $R=C(X, \mathbb{C})$. Then ${\mathbb M}_n((RG)[[X]])$ and
${\mathbb M}_n\left(\frac{(RG)[x]}{(x^k)}\right)$ are strongly
clean. In particular, $\mathbb{M}_{n}(R)$ is strongly
clean.\end{cor}
\begin{proof}By Theorem \ref{thm:3.1} and  \cite[Corollary 3.2]{YZ071}.
\end{proof}
\begin{cor} If $X$ is a discrete space, then
$\mathbb{M}_{n}( C(X, \mathbb{C}))$ is strongly $\pi$-regular
(hence, strongly clean).
\end{cor}
\begin{proof} Every discrete space is a {\rm P}-space relative to ${\mathbb C}$.  \end{proof}
\section*{Acknowledgments} The first author and the second were partially supported by NSERC of
Canada and the Initial Grant of Harbin Institute of Technology
respectively.


\begin{thebibliography}{00}
\bibitem{A02} F. Azarpanah, When is $C(X)$ a clean ring? {\it Acta Math. Hungar.}, {\bf 94} (1-2) (2002), 53-58.
\bibitem{Az54} G. Azumaya, Strongly $\pi$-regular rings, {\it J. Fac. Sci. Hokkaido Univ.}, {\bf 13} (1954), 34-39.
\bibitem{BDD051} G. Borooah, A. J. Diesl, and T. J. Dorsey, Strongly clean matrix rings over commutative local rings,
{\it J. Pure and Appl. Algebra}, {\bf 212} (1) (2008), 281-296.
\bibitem{BM88} W. D. Burgess and P. Menal, On strongly $\pi $-regular rings and homomorphisms into them, {\it Comm. Algebra}, {\bf 16} (1988), 1701-1725.
\bibitem{CY94} V. P. Camillo and H.-P. Yu, Exchange rings, units and idempotents, {\it Comm. Algebra}, {\bf 22} (12) (1994), 4737-4749.
\bibitem{CYZ052} J. Chen, X. Yang, and Y. Zhou, When is the $2\times2$ matrix ring over a commutative local ring strongly clean? {\it J. Algebra},
{\bf 301} (1) (2006), 280-293.
\bibitem{C85} P. M. Cohn, {\it Free Rings and Their Relations}, second edition, Academic Press, 1985.
\bibitem{E77} R. Engelking, {\it General Topology}, Heldermann Verlag Berlin, 1989.
\bibitem{FY071} L. Fan and X. Yang, Strongly clean property and stable range one of some rings, {\it Comm. Algebra}, to appear.
\bibitem{GJ76} L. Gillman and M. Jerison, {\it Rings of Continuous Functions}, Springer-Verlag, 1976.
\bibitem{H90} Y. Hirano, Some characterizations of $\pi$-regular
rings of bounded index, {\it Math. J. Okayama Univ.}, {\bf 32}
(1990), 90-101.
\bibitem{Hu74} T. W. Hungerford, {\it Algebra}, Springer-Verlag, New York, Heidelberg, Berlin, 1974.
\bibitem{Li07} Y. Li, On strongly clean matrix rings, {\it J. Algebra}, {\bf 312} (1) (2007), 397-404.
\bibitem{Mc84} B. R. McDonald, {\it Linear Algebra over Commutative Rings}, Marcel Dekker, 1984.
\bibitem{Mc03} W. W. McGovern, Clean semiprime $f$-rings with bounded inversion, {\it Comm. Algebra}, {\bf 31} (7) (2003), 3295-3304.
\bibitem{N77} W. K. Nicholson, Lifting idempotents and exchange rings, {\it Trans. Amer. Math. Soc.}, {\bf 229} (1977), 269-278.
\bibitem{N99} W. K. Nicholson, Strongly clean rings and Fitting's Lemma, {\it Comm. Algebra}, {\bf 27} (1999), 3583-3592.
\bibitem{WC04} Z. Wang and J. Chen, On two open problems about strongly clean rings, {\it Bull. Austral. Math. Soc.}, {\bf 70} (2004), 279-282.
\bibitem{YZ071} X. Yang and Y. Zhou, Some families of strongly clean rings, {\it Linear Algebra Appl.}, {\bf 425} (2007), 119-129.
\bibitem{YZ072} X. Yang and Y. Zhou, Strongly clean property of the $2\times 2$ matrix ring over a general local ring, {\it J. Algebra}, to appear.
\end{thebibliography}
\end{document}